\documentclass[12pt,reqno]{amsart}
 \usepackage[utf8]{inputenc}
 \usepackage[T1]{fontenc}
 \usepackage[usenames, dvipsnames]{color}
 \usepackage{ulem}
 
 \usepackage{dsfont, amsfonts, amsmath, amssymb,amscd, stmaryrd, latexsym, amsthm, dsfont}
 \usepackage[frenchb,english]{babel}
 \usepackage{enumerate}
 \usepackage{longtable}
 \usepackage{geometry}
 \usepackage{tikz}
 \usepackage{float}
 \usetikzlibrary{shapes,arrows}
 \usepackage{float}
 \usepackage{tabularx}
 \newtheorem{prop}{Proposition}
 \newtheorem{theorem}{Theorem}

 \newtheorem{rmk}{Remark}
 \newtheorem{lm}{Lemma}

 \def\ZZ{\mathbb{Z}}

\begin{document}
 	
 	\selectlanguage{english}
 	
 	\title[The rank of the $2$-class ...]{The rank of the $2$-class group of some fields with large degree}

 	\author[M. M. Chems-Eddin]{Mohamed Mahmoud CHEMS-EDDIN}
 	\address{Mohamed Mahmoud CHEMS-EDDIN: Mohammed First University, Mathematics Department, Sciences Faculty, Oujda, Morocco }
 	\email{2m.chemseddin@gmail.com}	


 	
 	\maketitle

 	\begin{abstract}
 		Let $n\geq 3$ be an integer and  $d$   an odd square-free integer.  We shall compute the rank of the $2$-class group of $L_{n,d}:=\mathbb{Q}(\zeta_{2^n},\sqrt{d})$, when all the prime divisors of $d$ are congruent to $\pm 3\pmod 8$ or 
 		$9\pmod{16}$.
 	\end{abstract}

 	\section{\textbf{Introduction}}
 	The explicit computation of the rank of the $2$-class group of a given  number field $K$ is one  of  the difficult problems of algebraic number theory, especially for fields with large degree.
 	For many years ago, several  authors studied this problem for number fields of degree  $2$ or  $4$  (cf. \cite{kaplan76,thomasParry,mccall19972}). The methods used therein  are not enough to deal with the same problem for number
 	fields with large degree, although recently some papers studied this question for some number fields of degree $2^n$  (cf. \cite{LiouYangXuZhang,chemsZkhnin2}).
 	Using the cyclotomic units and some results of the theory of the  cyclotomic $\mathbb{Z}_2$-extension, we extend these methods to compute the rank of the $2$-class group of some fields of degree $2^n$ of the form
 	$L_{n,d}:=\mathbb{Q}(\zeta_{2^n},\sqrt{d})$, where $d$ is an  odd square-free integer and  $n\geq 3$ is a positive integer.
 	
 	Let  $\mathds{k}:=\mathbb{Q}(\sqrt{d}, \sqrt{-1})$, $\mathbb{Q}(\sqrt{-2}, \sqrt{d})$ or $\mathbb{Q}(\sqrt{-2}, \sqrt{-d})$. Then the cyclotomic  $\mathbb{Z}_2$-extension of $\mathds{k}$ is 
 	$$\mathds{k}(\sqrt{2}) \subset \mathds{k}(\sqrt{2+\sqrt{2}}) \subset \mathds{k}(\sqrt{2+\sqrt{2+\sqrt{2}}})\subset ...$$
 	which coincides with the   tower   $L_{3,d}\subset L_{4,d} \subset...\subset  L_{n,d} \subset ... $
 	
 	The present work is a continuation of our previous work  \cite{chemsZkhnin2}, in which we computed the rank of the $2$-class group of $ L_{n,d}$ when the prime divisors  of $d$  are
 	congruent to $3$ or $5\pmod 8$. Thus, we compute the rank of the $2$-class group of $ L_{n,d}$ when the prime divisors of $d$ are congruent to    $\pm 3\pmod 8$ or $9\pmod{16}$. Furthermore, we give the rank of the $2$-class group of  $ L_{n,d}$ in terms  of that of
 	$ L_{4,d}$, when the prime divisors of $d$ are congruent to   $\pm 3\pmod 8$ or $\pm 7\pmod{16}$.
 	\section*{Notations}
 	The next notations will be used for the rest of this article:
 	{ \footnotesize \begin{enumerate}[\rm$\bullet$]
 			\item $d$:  An odd square-free integer,
 			\item $n$: A positive integer $\geq 3$,
 			\item $K_n$: $\mathbb{Q}(\zeta_{2^n})$,
 			\item $K_n^+$: The maximal real subfield of  $K_n$,
 			\item $L_{n,d}$:  $K_n(\sqrt{d})$,
 			\item $\mathcal{N}$: The norm map of the extension     $L_{n,d}/K_n$,
 			\item $E_{k}$:   The unit group of a number field $k$,
 			\item 	$\mathcal O_{k}$:  The ring of integers of a number field $k$,
 			\item $ {Cl}(k)$:  The class group of a number field  $k$,
 			\item $ {Cl}_2(k)$:  The $2$-class group of a number field  $k$,
 			\item $rank_2(Cl(L_{n,d}))$:  The rank of the $2$-class group of $L_{n,d}$,
 			\item $\mathfrak p$: A prime ideal of $K_n$,
 			\item $\left( \frac{\alpha,d}{\mathfrak p}\right)$:  The quadratic  norm residue symbol 	for $L_{n,d}/K_n$,
 			\item $\left( \frac{\alpha}{ \mathfrak{p}}\right)$: The quadratic power residue symbol.

 	\end{enumerate}}

 	\section{ \textbf{The preliminary results }}
 	Let us  first collect some results that will be used in what follows.
 	\begin{lm}[\cite{fukuda}]\label{lm fukuda}
 		Let $K/k$ be a $\mathbb{Z}_2$-extension, $k_n$ its    n-th layer  and $n_0$  an integer such that any prime of $K$ which is ramified in $K/k$ is totally ramified in $K/k_{n_0}$.
 		If there exists an integer $n\geq n_0$ such that $rank_2(  {Cl}(k_n))= rank_2( {Cl}(k_{n+1}))$, then
 		$rank_2( {Cl}(k_{m}))= rank_2( {Cl}(k_{n}))$ for all $m\geq n$.
 	\end{lm}

 	\begin{lm}[\text{\cite[Lamma 8.1, Corollary 4.13]{washington1997introduction}}]\label{lm cyclo units}~\
 		\begin{enumerate}[\rm 1.]
 			\item The cyclotomic units of   $K_n^+$ are generated by  $-1$ and $$\xi_{k,n}=\zeta_{2^n}^{(1-k)/2}\frac{1-\zeta_{2^n}^k}{ 1-\zeta_{2^n}},$$ where $k$ is an odd integer such that  $1< k< 2^{n-1}$.
 			\item The cyclotomic units of $K_n$ are generated by $\zeta_{ 2^{n}}$ and the cyclotomic units of   $K_n^+$ .
 			
 			\item The Hasse's index $Q$ of $K_n$ equals 1.
 		\end{enumerate}
 		
 	\end{lm}

 	\begin{lm}[\text{\cite{Gr}}]\label{ambiguous class number formula} Let $K/k$ be a quadratic extension. If the class number of $k$ is odd, then the rank of the $2$-class group of $K$ is given by
 		$$rank_2({Cl}(K))=t-1-e,$$
 		where  $t$ is the number of  ramified primes (finite or infinite) in the extension  $K/k$ and $e$ is  defined by   $2^{e}=[E_{k}:E_{k} \cap N_{K/k}(K^*)]$.
 	\end{lm}

 	\begin{rmk}\label{Rmk 1}
 		Note that a unit $u$ of $K_n$ is a norm  in $L_{n,d}/K_n$ if and only if $\left( \frac{u,d}{\mathfrak p}\right)=1$, for all prime ideal   $\mathfrak p$ of $K_n$ ramified in $L_{n,d}$.
 	\end{rmk}


 	\hspace{-0.4cm}\begin{minipage}{10.5cm}
 		Next, we need to characterize ideals of  $K_n=\mathbb{Q}(\zeta_{2^n})$ that ramify in $L_{n, d}=K_n(\sqrt{d})$.  Let $n\geq 3$ and  $d$ be an  odd square-free integer, then $d$ is congruent to $1$ or $3 \pmod 4$. So  $2$     is unramified in either $\mathbb Q(\sqrt{ d})$ or  $\mathbb Q(\sqrt {-d})$. Thus, the ramification index  of $2$ in $L_{n,d}$ is  strictly inferior to $2^n$. As  $2$ is totally ramified in $K_n:=\mathbb Q(\zeta_{2^n})$, so the prime ideal
 		of $K_n$ lying over $2$ is unramified in $L_{n,d}$, as otherwise  the ramification index    of $2$ in $L_{n,d}$ will be $2^n$, which  is absurd. Hence  we prove the following result:
 	\end{minipage}
 	\fbox{\begin{minipage}{3.6cm}
 			{	\footnotesize
 				
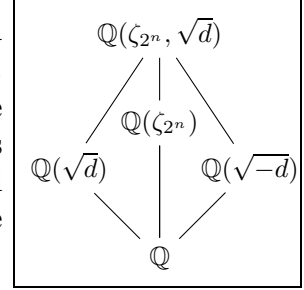
\hspace{0.5cm}\begin{tikzpicture} [scale=1.2]
 				\node (Q)  at (0,0) {$\mathbb Q$};
 				\node (d)  at (-1,1) {$\mathbb Q(\sqrt d)$};
 				\node (-d)  at (1,1) {$\mathbb Q(\sqrt {-d})$};
 				\node (zeta)  at (0,1.5) {$\mathbb Q(\zeta_{2^n})$};
 				\node (zeta d)  at (0,2.5) {$\mathbb Q(\zeta_{2^n},\sqrt d)$};
 				\draw (Q) --(d)  node[scale=0.4,midway,below right]{};
 				\draw (Q) --(-d)  node[scale=0.4,midway,below right]{};
 				\draw (Q) --(zeta)  node[scale=0.4,midway,below right]{};
 				\draw (Q) --(zeta)  node[scale=0.4,midway,below right]{};
 				\draw (zeta) --(zeta d)  node[scale=0.4,midway,below right]{};
 				\draw (d) --(zeta d)  node[scale=0.4,midway,below right]{};
 				\draw (-d) --(zeta d)  node[scale=0.4,midway,below right]{};
 				\end{tikzpicture}}
 	\end{minipage}}
 	
 	\begin{lm}\label{lm ramified primes of L/K}
 		Let $d$ be an odd square-free integer. Then a prime ideal $\mathfrak p$ of $K_n$ is ramified in  $L_{n,d}/K_n$ if and only if it divides
 		$d.$
 	\end{lm}		
 	
 	\begin{prop}\label{Number of primes dividing d}
 		Let $n\geq5$ and    $p$ be a rational prime. Then $p$ decomposes into    $4$ primes of $K_n$  if and only if   $p\equiv 7\text{ or } 9 \pmod {16}$.
 	\end{prop}
  \begin{proof}
 	Let  $K_\infty$    denote  the    cyclotomic $\mathbb{Z}_2$-extension of $K $.
 	  Note that  $\mathrm{Gal}(K_\infty:K)\simeq \ZZ_2$. Hence the decomposition field of a prime above $p$ must be some $K_n$. Therefore if a prime $\mathfrak{p}$ of $K_n$ is inert in $K_{n+1}$, then 
 	  $\mathfrak{p}$ is inert in  $K_\infty/K_n$. Then Proposition 1 is from the easily verified fact that a prime $p$ decomposes into $4$ primes in $K_4=\mathbb{Q}(\zeta_{16})$ and $K_5=\mathbb{Q}(\zeta_{32})$ if and only if $p\equiv 7,9 \pmod {16}$.
  \end{proof}

 	\begin{rmk}\label{Rmq decoposition in K4}Let   $p\equiv 7\text{ or } 9 \pmod {16}$ be a prime.  We have
 		\begin{enumerate}[\rm 1.]
 			\item If $p\equiv  9 \pmod {16}$, then $p$ decomposes into product  of four prime ideals  in $K_3$ and $K_4$.
 			\item If $p\equiv  7 \pmod {16}$, then $p$ decomposes into product  of  two prime ideals in $K_3$  and into product  of four prime ideals in  $K_4$.
 		\end{enumerate}
 	\end{rmk}

 	\begin{prop}[\cite{chemsZkhnin2}]\label{prop deco to 2}
 		Let $m\geq4$ and    $p$ be a rational prime. Then $p$ decomposes into the product of  $2$ primes of $K_m$  if and only if   $p\equiv3$ or $5 \pmod 8$.
 	\end{prop}
 	
 	Now we shall   do some computations:
 	
 	\begin{lm}[\cite{chemsZkhnin2}]\label{lm1 computations}
 		Let $n\geq 3$ be a positive integer  and $p$ be a prime number. Let 
 		$\mathfrak p_{K_{n}}$ denote  a prime ideal of $K_n$ above $p$. We have
 		
 		\begin{enumerate}[\rm 1.]

 			\item If    $  p\equiv 5  \pmod{8}$.  Then
 			$$\left( \frac{\zeta_{2^n},p}{\mathfrak p_{K_{n}}}\right)=-1 \text{ and }\left( \frac{\xi_{k,n},p}{\mathfrak p_{K_{n}}}\right)=1.$$
 			\item If  $p\equiv 3  \pmod{8}$. Then
 			$$\left( \frac{\zeta_{2^n},p}{\mathfrak p_{K_{n}}}\right)=-1 \text{ and }\left( \frac{\xi_{k,n},p}{\mathfrak p_{K_{n}}}\right)=\left\{  \begin{array}{ccc}
 			-1, &\text{ if } k \equiv \pm 3 \pmod 8& \\
 			1,& \text{ elsewhere. } &
 			\end{array} \right.$$
 			
 		\end{enumerate}
 	\end{lm}

 	\begin{lm}\label{lm compo equiv 9 mod 16}
 		Let $n\geq 3$ be an integer and  $p$  a  prime  congruent     to $9\pmod{16}$, $\mathfrak p_{K_{n}}$ a prime ideal of $K_n$ dividing $p$.	
 		$$\left( \frac{\zeta_{2^n},p}{\mathfrak p_{K_{n}}}\right)=-1\text{ and }\left( \frac{\xi_{k,n},p}{\mathfrak p_{K_{n}}}\right)= \left\{  \begin{array}{cl}
 		-\left(\frac{2}{p}\right)_4,&\text{ if } k\equiv \pm3 \pmod 8 \\
 		1,&\text{elsewhere}. \\
 		\end{array} \right.$$
 	\end{lm}
 	\begin{proof}
 		For all $n\geq 3$, the  prime $p$ decomposes  into product  of four prime ideals of $K_n$ (see Proposition \ref{Number of primes dividing d}),
 		denote by $\mathfrak p_{K_{n}}$ one of them (such that $\mathfrak p_{K_{n-1}}\subset \mathfrak p_{K_{n}}$).  We have  $\zeta_{2^{n}}^2=\zeta_{2^{{n-1}}}$, so the minimal polynomial of $\zeta_{2^{n}}$ over $K_{{n-1}}$ is $X^2-\zeta_{n-1}$ and $N_{K_{n}/K_{n-1}}(\zeta_{2^{n}})=-\zeta_{2^{n-1}}$.
 		Thus
 		$$\begin{array}{ll}
 		\left( \frac{\zeta_{2^{n}},p}{\mathfrak p_{{K_{n}}}}\right)=\left( \frac{-\zeta_{ 2^{n-1}},p}{\mathfrak p_{{K_{n-1}}}}\right)=\left( \frac{\zeta_{2^{n-1}},p}{\mathfrak p_{{K_{n-1}}}}\right)=...=\left( \frac{\zeta_{8},p}{\mathfrak p_{{K_{3}}}}\right)=(-1)^{\frac{p-1}{8}} =-1(\text{see \cite{chemsZkhnin1}}).
 		\end{array}$$
 		$$\begin{array}{ll}
 		\left( \frac{\xi_{k,n},p}{\mathfrak p_{K_{n}}}\right)
 		&=\left( \frac{\zeta_{2^n}^{(1-k)/2},p}{\mathfrak p_{K_{n}}}\right)\left( \frac{\frac{1-\zeta_{2^n}^k}{1-\zeta_{2^n}},p}{\mathfrak p_{K_{n}}}\right)\\
 		&=(-1)^{(1-k)/2}\left( \frac{(1-\zeta_{2^n}^k)(1-\zeta_{2^n}),p}{\mathfrak p_{K_{n}}}\right)\\
 		&=(-1)^{(1-k)/2} \left( \frac{1-\zeta_{8}^k,p}{\mathfrak p_{{K_{3}}}}\right)\left( \frac{1-\zeta_{8},p}{\mathfrak p_{{K_{3}}}}\right)\\
 		&=(-1)^{(3-k)/2}\left( \frac{\zeta_{8}^{-1},p}{\mathfrak p_{{K_{3}}}}\right) \left( \frac{1-\zeta_{8}^k,p}{\mathfrak p_{{K_{3}}}}\right)\left( \frac{1-\zeta_{8},p}{\mathfrak p_{{K_{3}}}}\right)\nonumber\\
 		&\vspace{0.2cm}= \left\{  \begin{array}{cc}
 		\left( \frac{ \varepsilon_2,p}{\mathfrak p_{K_3}}\right),&\hspace{-1.55cm}\text{ if } k\equiv 3 \pmod 8 \\
 		\left( \frac{1+\zeta_{8},p}{\mathfrak p_{{K_{3}}}}\right)\left( \frac{1-\zeta_{8},p}{\mathfrak p_{{K_{3}}}}\right),&\text{ if } k\equiv 5 \pmod 8 \;\;\;\nonumber \\
 		-\left( \frac{1-\zeta_{8}^{-1},p}{\mathfrak p_{{K_{3}}}}\right)\left( \frac{1-\zeta_{8},p}{\mathfrak p_{{K_{3}}}}\right),&\text{ if } k\equiv 7 \pmod 8 \\
 		\left( \frac{1-\zeta_{8},p}{\mathfrak p_{{K_{3}}}}\right)\left( \frac{1-\zeta_{8},p}{\mathfrak p_{{K_{3}}}}\right),&\text{ if } k\equiv 1 \pmod 8 \\
 		\end{array} \right.\\
 		&\vspace{0.2cm}= \left\{  \begin{array}{cc}
 		-\left(\frac{2}{p}\right)_4,&\text{ if } k\equiv 3 \pmod 8 \;(\text{see \cite{chemsZkhnin1}})\vspace{0.2cm}\\
 		\left( \frac{1-i,p}{\mathfrak p_{{K_{3}}}}\right),&\text{ if } k\equiv 5 \pmod 8 \;\;\;\nonumber \vspace{0.2cm}\\
 		-\left( \frac{2-\sqrt{2},p}{\mathfrak p_{{K_{3}}}}\right),&\text{ if } k\equiv 7 \pmod 8 \vspace{0.2cm}\\
 		1,&\text{ if } k\equiv 1 \pmod 8. \\
 		\end{array} \right.\\
 		
 		\end{array}$$
 		
 		Using the proof of \cite[Lemma 3.4]{chemsZkhnin1}  we get \\
 		{\centering $\left( \frac{1-i,p}{\;\mathfrak p_{{K_{3}}}}\right)=\left( \frac{1-i}{\;\mathfrak p_{{K_{3}}}}\right)=\left( \frac{1+i}{\;\mathfrak p_{{K_{3}}}}\right)=\left( \frac{\zeta_8,p}{\;\mathfrak p_{{K_{3}}}}\right)\left( \frac{\sqrt{2},p}{\;\mathfrak p_{{K_{3}}}}\right)=\left( \frac{\zeta_8,p}{\;\mathfrak p_{{K_{3}}}}\right)(-1)^{\frac{p-1}{8}}\left( \frac{\varepsilon_2,p}{\;\mathfrak p_{{K_{3}}}}\right)=\left( \frac{\varepsilon_2,p}{\;\mathfrak p_{{K_{3}}}}\right)=-\left(\frac{2}{p}\right)_4.$ }\\
 		Similarly we have 
 		{\centering $ \left( \frac{2-\sqrt{2},p}{\mathfrak p_{{K_{3}}}}\right)= \left( \frac{2+\sqrt{2}}{\mathfrak p_{{K_{3}}}}\right)= \left( \frac{\sqrt{2},p}{\mathfrak p_{{K_{3}}}}\right)\left( \frac{\varepsilon_2,p}{\mathfrak p_{{K_{3}}}}\right)=(-1)^{\frac{p-1}{8}}\left( \frac{\varepsilon_2,p}{\;\mathfrak p_{{K_{3}}}}\right)\left(\frac{\varepsilon_2,p}{\mathfrak p_{{K_{3}}}}\right)=-1$ }. Which achieves the proof.
 	\end{proof}

 	\section{\textbf{The main results}}
 	The authors of \cite{chemsZkhnin2} computed  the rank of the $2$-class group of $L_{n,d}$, when the prime divisors of $d$ are congruent to $3$ or $5 \pmod 8$.
 	In this section we shall compute the rank of the $2$-class group of $L_{n,d}$, then the prime divisors of $d$ are congruent to $\pm 3\pmod 8$ or $9\pmod{16}$.

 	\begin{theorem}\label{thm main theorem}
 		Let $n\geq 3$ and $d>2$ be an odd composite   square-free integer of prime divisors congruent to $\pm 3\pmod 8$ or $9\pmod {16}$. Let $r$ denote the number of prime divisors of $d$ which are congruent to $3$ or $5\pmod 8$  and $q$ the number of those which are congruent to $9\pmod{16}$. Set  $t=4q+2r$. We have  
 		\begin{enumerate}[\rm1.]
 			\item   If there are two primes $p_1 $ and   $p_2$ dividing $d$ such that  $p_1\equiv -p_2\equiv 5\pmod 8$, then
 			$rank_2(Cl(L_{n,d}))=t-3.$
 			
 			\item If $d$  is divisible by a prime congruent to $3\pmod 8$ and none of the primes $p|d$ is congruent to $5\pmod 8$, then $rank_2(Cl(L_{n,d}))=t-2$ or $t-3$. More precisely,
 			$rank_2(Cl(L_{n,d}))= t-3$ if and only if there is a prime $p\equiv 1\pmod 8$ dividing $d$ such that  $  \left(\frac{2}{ p}\right)_4=-1$.
 			
 			\item If $d$  is divisible by a prime congruent to $5\pmod 8$ and none of the primes $p|d$ is congruent to $3\pmod 8$, then $rank_2(Cl(L_{n,d}))=t-2$ or $t-3$. More precisely,
 			$rank_2(Cl(L_{n,d}))= t-3$ if and only if there is a prime $p\equiv 1\pmod 8$ dividing $d$ such that  $ \left(\frac{2}{ p}\right)_4=1$.
 			\item If all the primes $p|d$ are congruent to $9\pmod  {16}$, then $rank_2(Cl(L_{n,d}))=4q-2$ or $4q-3$. More precisely, $ rank_2(Cl(L_{n,d}))= 4q-3$ if and only if there are two prime  divisors $p_1$ and $p_2$ of $d$ such that   $
 			\left(\frac{2}{ p_1}\right)_4=1$ and
 			$\left(\frac{2}{ p_2}\right)_4=-1$.
 		\end{enumerate}
 	\end{theorem}	
 	\begin{proof}
 		We shall firstly  prove the items of  the above theorem assuming that $n\in \{3,4,5\}$. 	We have $h(K_n^+)=1$ (see \cite{masley}). So by \cite[Theorem 8.2]{washington1997introduction}, the unit group of $K_n^+$ is  generated by  $-1$ and the cyclotomic units  $\xi_{k,n}$, for odd integers $k$ such that  $1< k< 2^{n-1}$.
 		Thus, by Lemma \ref{lm cyclo units},  we have 
 		$$E_{K_n}=\langle \zeta_{ 2^{n}}, \xi_{k,n}, \text{ with  $k$ is an odd integer such that } 1< k< 2^{n-1}\rangle.$$
 		By   Lemma \ref{ambiguous class number formula}, Proposition \ref{Number of primes dividing d} and Lemma \ref{lm ramified primes of L/K}
 		we have  $rank_2(Cl(L_{n,d}))=t-1-e_{n,d}$, where $e_{n,d}$ is defined by $(E_{K_n}:E_{K_n}\cap  \mathcal{N}(L_{n,d}))=2^{ e_{n,d}}$.
 		We shall determine the classes representing $E_{K_n}/(E_{K_n}\cap  \mathcal{N}(L_{n,d}))$.
 		Let   $\alpha\in E_{K_n} $.    $\overline{\alpha}$ denotes the class of $\alpha$  in  $E_{K_n}/(E_{K_n}\cap  \mathcal{N}(L_{n,d}))$. Let $p$ be a prime dividing $d$ and  $\mathfrak p_{K_{n}}$ be a prime ideal of $K_n$     lying over $p$.
 		We have
 		\begin{eqnarray*}
 			\left( \frac{\alpha,d}{\mathfrak p_{K_{n}}}\right)=\prod_{q|d} \left( \frac{\alpha,d}{\mathfrak p_{K_{n}}}\right)&=&\left(\frac{\alpha,p}{\mathfrak p_{K_{n}}}\right)\prod_{q|d\; \text{ and } \;q\not= p} \left( \frac{\alpha,q}{\mathfrak p_{K_{n}}}\right)\\
 			&=&\left(\frac{\alpha,p}{\mathfrak p_{K_{n}}}\right)\prod_{q|d\;  \text{ and } \;q\not= p}  \left( \frac{q}{ \;\mathfrak p_{K_{n}}}\right)^{0}
 			=\left(\frac{\alpha,p}{\mathfrak p_{K_{n}}}\right).
 		\end{eqnarray*}

 		Note that the units $\xi_{k,n}$ for $k\equiv \pm 1\pmod 8$ are  norms in $L_{n,d}/K_n$ so we will disregard them.   Let $k$ and $k'$ denote any two positive integers such that $k, k'\equiv \pm 3\pmod 8.$
 		\begin{enumerate}[\rm1.]
 			
 			\item  
 			By Lemmas \ref{lm1 computations} and \ref{lm compo equiv 9 mod 16}, we have $\zeta_{ 2^{n}}$ and $\xi_{k,n}$  are not norms in $L_{n,d}/K_n$. Furthermore 
 			$$   \left\{\begin{array}{cl}
 			\left( \frac{\zeta_{ 2^{n}}\xi_{k,n},p}{\mathfrak p_{K_{n}}}\right)=-1, & \text{ if }    p\equiv 5\pmod 8,\vspace{0.3cm}\\
 			\left( \frac{ \xi_{k,n}\xi_{k',n},p}{\mathfrak p_{K_{n}}}\right)=1, &  \text{ for all prime } p \text{ dividing } d.\\
 			
 			\end{array}\right.$$

 			Thus $\overline{\zeta_{2^n}}\not=  \overline{\xi_{k,n}}$ and  $\overline{\xi_{k,n}}=\overline{\xi_{k',n}}$. Hence	$E_{K_n}/(E_{K_n}\cap  \mathcal{N}(L_{n,d}))=\{\overline{1},\overline{\zeta_{2^n}},\overline{\xi_{k,n}}, \overline{\zeta_{2^n}\xi_{k,n}}\}$. It follows that $rank_2(Cl(L_{n,d}))=t-3$.
 			
 			\item 
 			By Lemmas \ref{lm1 computations} and \ref{lm compo equiv 9 mod 16}, we have $\zeta_{ 2^{n}}$ and $\xi_{k,n}$  are not norms in $L_{n,d}/K_n$. Furthermore 
 			$$   \left\{\begin{array}{cl}
 			\left( \frac{\zeta_{ 2^{n}}\xi_{k,n},p}{\mathfrak p_{K_{n}}}\right)=\left(\frac{2}{p}\right)_4, & \text{ if }    p\equiv 3\pmod 8,\vspace{0.3cm}\\
 			\left( \frac{ \xi_{k,n}\xi_{k',n},p}{\mathfrak p_{K_{n}}}\right)=1, &  \text{ for all prime } p \text{ dividing } d.\\
 			\end{array}\right.$$
 			Then $\overline{\xi_{k,n}}=\overline{\xi_{k',n}}$ and 	  $\overline{\zeta_{2^n}}\not=  \overline{\xi_{k,n}}$ if and only if $\left(\frac{2}{p}\right)_4=-1$. Thus,   	$E_{K_n}/(E_{K_n}\cap \mathcal{N}(L_{n,d}))=\{\overline{1},\overline{\zeta_{2^n}}\}$, if  $\left(\frac{2}{p}\right)_4=1$ and 	$E_{K_n}/(E_{K_n}\cap \mathcal{N}(L_{n,d}))=\{\overline{1},\overline{\zeta_{2^n}},\overline{\xi_{k,n}}, \overline{\zeta_{2^n}\xi_{k,n}}\}$ if not. So the second item.
 			
 			\item By Lemmas \ref{lm1 computations} and \ref{lm compo equiv 9 mod 16}, we have $\zeta_{ 2^{n}}$  are not norm  in $L_{n,d}/K_n$  and $\xi_{k,n}$ is not norm in
 			$L_{n,d}/K_n$ is and only if $\left(\frac{2}{p}\right)_4=1$. So with similar discussion as above, one shows that 
 			$E_{K_n}/(E_{K_n}\cap \mathcal{N}(L_{n,d}))=\{\overline{1},\overline{\zeta_{2^n}}  \}$ if $\left(\frac{2}{p}\right)_4=-1$ and  $E_{K_n}/(E_{K_n}\cap \mathcal{N}(L_{n,d}))=\{\overline{1},\overline{\zeta_{2^n}},\overline{\xi_{k,n}}, \overline{\zeta_{2^n}\xi_{k,n}}\}$ if not. So the results.
 			
 			\item \begin{enumerate}[\rm 	$\bullet$]
 				\item 	Assume that  for all prime $p$ dividing $d$ we have   $\left(\frac{2}{ p}\right)_4=1$, then 	$\left( \frac{\zeta_{ 2^{n}}\xi_{k,n},p}{\mathfrak p_{K_{n}}}\right)=\left( \frac{\zeta_{ 2^{n}} ,p}{\mathfrak p_{K_{n}}}\right)\left( \frac{ \xi_{k,n},p}{\mathfrak p_{K_{n}}}\right)=1$.  It follows that $\overline{\zeta_{2^n}}=\overline{\xi_{k,n}}$. So
 				$E_{K_n}/(E_{K_n}\cap \mathcal{N}(L_{n,d}))=\{\overline{1},\overline{\zeta_{2^n}}\}$ and $e_{n,d}=1$. 
 				\item  If  for all the primes $p$ dividing $d$ we have $\left(\frac{2}{ p}\right)_4=-1$, then $\xi_{k,n}$ is norm in $L_{n,d}/K_n$.	So $E_{K_n}/(E_{K_n}\cap \mathcal{N}(L_{n,d}))=\{\overline{1},\overline{\zeta_{2^n}}\}$ and $e_{n,d}=1$. 
 				\item Suppose now there are two primes $p_1$ and $p_2$ dividing  $d$  such that $
 				\left(\frac{2}{ p_1}\right)_4=1$ and
 				$\left(\frac{2}{ p_2}\right)_4=-1$.
 				We have 	$\left( \frac{\xi_{k,n},p_1}{\mathfrak p_{K_{n}}}\right)=\left( \frac{\zeta_{ 2^{n}}\xi_{k,n},p_2}{\mathfrak p_{K_{n}}}\right)=-1$. Thus $\overline{\zeta_{2^n}}\not=\overline{\xi_{k,n}}$ and $\overline{\xi_{k,n}}\not=\overline{1}$. We infer that
 				$E_{K_n}/(E_{K_n}\cap \mathcal{N}(L_{n,d}))=\{\overline{1},\overline{\zeta_{2^n}},\overline{\xi_{k,n}}, \overline{\zeta_{2^n}\xi_{k,n}}\}$ and $e_{n,d}=2$. So the third item.
 			\end{enumerate}
 		\end{enumerate}	
 		Thus we proved the theorem for $n\in \{3,4,5\}$.\\
 		Since  $rank_2(Cl(L_{3,d}))=rank_2(Cl(L_{4,d}))$, then Lemma \ref{lm fukuda} achieves the proof.

 	\end{proof}	
 	
 	We similarly get the following result:
 	\begin{theorem}\label{thm 3}
 		Let $n\geq 3$ be a positive integer and let $p$ denote a prime such that   $p\equiv 9 \pmod{16}$. Then 
 		$$ rank_2(Cl(L_{n,p}))= 2.$$	
 	\end{theorem}

 	\section{\textbf{Appendix}}
 	In this appendix,    we give the rank of the $2$-class group of  $ L_{n,d}$ according to  that of
 	$ L_{4,d}$, when the prime divisors of $d$ are congruent to $\pm 3\pmod 8$ or $\pm 7\pmod{16}$.
 	\begin{lm}\label{lm computations 7 pmod}
 		Let $n\geq 4$ be an integer and $p$  a   prime integer  congruent to $7\pmod{16}$. Then for all    prime ideal $\mathfrak p_{K_{n}}$ of $K_n$ dividing $p$, we have
 		$$\left( \frac{\zeta_{2^n},p}{\mathfrak p_{K_{n}}}\right)=\left( \frac{\zeta_{16},p}{\mathfrak p_{K_{4}}}\right) \text{ and }\left( \frac{\xi_{k,n},p}{\mathfrak p_{K_{n}}}\right)= \left\{\begin{array}{cc}
 		\left( \frac{\xi_{3,4},p}{\mathfrak p_{K_{4}}}\right), & \text{ if }     k\equiv 3 \pmod{16}\\
 		\left( \frac{\xi_{5,4},p}{\mathfrak p_{K_{4}}}\right), & \text{ if }     k\equiv 5  \pmod {16}\\
 		\left( \frac{\xi_{7,4},p}{\mathfrak p_{K_{4}}}\right), & \text{ if }     k\equiv 7 \pmod{16}\\
 		1, & \text{ if }     k\equiv 1  \pmod{16}.\\
 		\end{array}\right.$$	
 		
 		And there   are $\varepsilon_1$, $\varepsilon_2$, $\varepsilon_3\in \{-1,1\}$ such that for all    prime ideal $\mathfrak p_{K_{n}}$ of $K_n$ dividing $p$ we have

 		$$ \left( \frac{\xi_{k,n},p}{\mathfrak p_{K_{n}}}\right)= \left\{\begin{array}{cl}
 		\varepsilon_1\left( \frac{\xi_{3,4},p}{\mathfrak p_{K_{4}}}\right), & \text{ if }     k\equiv  11 \pmod{16}\\
 		\varepsilon_2\left( \frac{\xi_{5,4},p}{\mathfrak p_{K_{4}}}\right), & \text{ if }     k\equiv  13 \pmod {16}\\
 		\varepsilon_3\left( \frac{\xi_{7,4},p}{\mathfrak p_{K_{4}}}\right), & \text{ if }     k\equiv  15 \pmod{16}\\
 		1, & \text{ if }     k\equiv 9 \pmod{16}.\\
 		\end{array}\right.$$

 		%
 	\end{lm}
 	\begin{proof}By Remark \ref{Rmq decoposition in K4} and Proposition \ref{Number of primes dividing d}, there are four prime ideals of $K_4$ lying over $p$, and these primes are inert in $K_n$ for all $n\geq 5$.
 		Since the minimal polynomial of $\zeta_{2^{n}}$ over $K_{{n-1}}$ is $X^2-\zeta_{n-1}$, then $N_{K_{n}/K_{n-1}}(\zeta_{2^{n}})=-\zeta_{2^{n-1}}$. We have \\
 		$$\left( \frac{\zeta_{2^n},p}{\mathfrak p_{K_{n}}}\right)=\left( \frac{\zeta_{2^n},p}{\mathfrak p_{K_{n}}}\right)=\left( \frac{\zeta_{{2^{n-1}}},p}{\mathfrak p_{{K_{n-1}}}}\right)=...=\left( \frac{\zeta_{16},p}{\mathfrak p_{{K_{4}}}}\right),$$
 		
 		and	
 		$$\left( \frac{1-\zeta_{{2^n}}^k,p}{\mathfrak p_{K_{n}}}\right)=\left( \frac{N_{K_n/K_{n-1}}(1-\zeta_{2^n}^k),p}{\mathfrak p_{{K_{n-1}}}}\right)=...=\left( \frac{1-\zeta_{16}^k,p}{\mathfrak p_{{K_{4}}}}\right).$$

 		\begin{eqnarray}
 		\left( \frac{\xi_{k,n},p}{\mathfrak p_{K_{n}}}\right)\nonumber
 		&=&\left( \frac{\zeta_{2^n}^{(1-k)/2},p}{\mathfrak p_{{K_{n}}}}\right)\left( \frac{\frac{1-\zeta_{2^n}^k}{1-\zeta_{2^n}},p}{\mathfrak p_{{K_{n}}}}\right)\nonumber\\
 		&=&\left( \frac{\zeta_{2^n},p}{\mathfrak p_{{K_{n}}}}\right)^{(1-k)/2}\left( \frac{(1-\zeta_{2^n}^k)(1-\zeta_{2^n}),p}{\mathfrak p_{{K_{n}}}}\right)\nonumber\\
 		&=&\left( \frac{\zeta_{2^n},p}{\mathfrak p_{{K_{n}}}}\right)^{(1-k)/2}\left( \frac{1-\zeta_{2^n}^k,p}{\mathfrak p_{{K_{n}}}}\right)\left( \frac{1-\zeta_{2^n},p}{\mathfrak p_{{K_{n}}}}\right)\nonumber\\
 		&=&\left( \frac{\zeta_{16},p}{\mathfrak p_{{K_{4}}}}\right)^{(1-k)/2} \left( \frac{1-\zeta_{16}^k,p}{\mathfrak p_{{K_{4}}}}\right)\left( \frac{1-\zeta_{16},p}{\mathfrak p_{{K_{4}}}}\right).\nonumber\end{eqnarray}
 		We consider the following cases
 		\begin{enumerate}[\rm $\bullet$]
 			\item If $k\equiv1,3,5\text{ or }7\pmod{16}$, then we have directly the result.
 			\item If $k\equiv 11\pmod{16}$, then $\zeta_{16}^k=e^{\frac{11\pi}{8}i}=-e^{\frac{3\pi}{8}i}=-\zeta_{16}^3$. Thus	
 			\begin{eqnarray*}
 				\left( \frac{\xi_{k,n},p}{\mathfrak p_{K_{n}}}\right)&=&	\left( \frac{\zeta_{16},p}{\mathfrak p_{{K_{n}}}}\right)^{(1-k)/2} \left( \frac{1+\zeta_{16}^3,p}{\mathfrak p_{{K_{4}}}}\right)\left( \frac{1-\zeta_{16},p}{\mathfrak p_{{K_{4}}}}\right)\\
 				&=&\left( \frac{\zeta_{16},p}{\mathfrak p_{{K_{n}}}}\right)^{(1-k)/2} \left( \frac{1+\zeta_{16}^3,p}{\mathfrak p_{{K_{4}}}}\right)\left( \frac{1-\zeta_{16},p}{\mathfrak p_{{K_{4}}}}\right)\left( \frac{1-\zeta_{16}^3,p}{\mathfrak p_{{K_{4}}}}\right)\left( \frac{1-\zeta_{16}^3,p}{\mathfrak p_{{K_{4}}}}\right)\\
 				&=&\left( \frac{(1+\zeta_{16}^3)(1-\zeta_{16}^3),p}{\mathfrak p_{{K_{4}}}}\right)\left( \frac{\xi_{3,4},p}{\mathfrak p_{K_{4}}}\right)=\left( \frac{1-\zeta_{8}^3,p}{\mathfrak p_{{K_{4}}}}\right)\left( \frac{\xi_{3,4},p}{\mathfrak p_{K_{4}}}\right).
 			\end{eqnarray*}
 			Since $\mathfrak p_{K_{3}}$  is totally decomposed in $K_4$, then  $\left( \frac{\xi_{k,n},p}{\mathfrak p_{K_{n}}}\right)=\left( \frac{1-\zeta_{8}^3,p}{\mathfrak p_{{K_{3}}}}\right)\left( \frac{\xi_{3,4},p}{\mathfrak p_{K_{4}}}\right)$.
 			We have $p$ decomposes into product  of two prime ideals, $\mathfrak p_{K_{3}}$ and $\mathfrak p_{K_{3}}'$ of $K_3$, then by the product formula, we have :
 			
 			$$\left( \frac{1-\zeta_{8}^3,p}{\mathfrak p_{{K_{3}}}}\right)\left( \frac{1-\zeta_{8}^3,p}{\mathfrak p_{{K_{3}}}'}\right) \left( \frac{1-\zeta_{8}^3,p}{\mathfrak q_{{K_{3}}}}\right)=1,$$
 		where  $\mathfrak q_{{K_{3}}}$ is the prime ideal of $K_3$ above $2$. By Remark \ref{Rmk 1} and 
 		Lemma \ref{lm ramified primes of L/K}
 		 we easily deduce that  $\left( \frac{1-\zeta_{8}^3,p}{\mathfrak q_{{K_{3}}}}\right)=1$. Thus	$\left( \frac{1-\zeta_{8}^3,p}{\mathfrak p_{{K_{3}}}}\right)=\left( \frac{1-\zeta_{8}^3,p}{\mathfrak p_{{K_{3}}}'}\right)=\varepsilon_1$
 			and
 			$	\left( \frac{\xi_{k,n},p}{\mathfrak p_{K_{n}}}\right)=\varepsilon_1\left( \frac{\xi_{3,4},p}{\mathfrak p_{K_{4}}}\right).$
 			\item Similarly, we show the relationships between the other norm residue   symbols.	\end{enumerate}	\end{proof}

 	\begin{theorem}
 		Let  $d$ be a square-free integer such that the prime    divisors of $d$ are congruent to $\pm 3\pmod 8$ or $\pm 7\pmod{16}$. Then, for all positive integer  $n\geq 4$
 		, we have
 		$$rank_2(Cl(L_{n,d}))=rank_2(Cl(L_{4,d})).$$	
 	\end{theorem}
 	\begin{proof}
 		Suppose that $n\in \{4,5\}$. By Lemmas \ref{lm1 computations},  \ref{lm compo equiv 9 mod 16} and \ref{lm computations 7 pmod}, we have  $e_{n,d}=e_{4,d} $. By 
 		  Propositions \ref{Number of primes dividing d} and \ref{prop deco to 2} the number of prime divisors of 
 		$d$ in $K_n$ is the same. Then  $rank_2(Cl(L_{5,d}))=rank_2(Cl(L_{4,d}))$ (see Lemma \ref{ambiguous class number formula}). Hence, Lemma \ref{lm fukuda} completes the proof. 
 	\end{proof}

 	\section*{\textbf{Acknowledgment}}
 	I would like to express my gratitude to my  professor Abdelkader Zekhnini for his support and  remarks during the preparation of this paper. My thanks are also due to the referee for his/her careful reading 
 	   and  constructive comments.

 \end{document}